\documentclass[12pt]{article}
\usepackage{geometry,amsthm,amssymb,amsmath,enumerate,float,cite,algorithm2e,verbatim,tikz}
\usetikzlibrary{positioning,arrows}
\tikzset{
  state/.style={circle,draw,minimum size=6ex},
  arrow/.style={-latex, shorten >=1ex, shorten <=1ex}}

\newtheorem{theorem}{Theorem}
\newtheorem{conjecture}{Conjecture}

\title{Domination versus edge domination}
\author{
Julien Baste$^{1,}$\footnote{Funded by the Deutsche Forschungsgemeinschaft (DFG, German Research Foundation) - 388217545.} \and
Maximilian F\"{u}rst$^1$ \and
Michael A. Henning$^2$ \and
Elena Mohr$^1$ \and
Dieter Rautenbach$^1$}
\date{}

\begin{document}

\maketitle

{\small
\begin{center}
$^1$
Institute of Optimization and Operations Research, Ulm University, Germany,
\texttt{\{julien.baste,maximilian.fuerst,elena.mohr,dieter.rautenbach\}@uni-ulm.de}\\[3mm]
$^2$
Department of Mathematics and Applied Mathematics, University of Johannesburg,
Auckland Park, 2006, South Africa,
\texttt{mahenning@uj.ac.za}
\end{center}
}

\begin{abstract}
We propose the conjecture that the domination number $\gamma(G)$
of a $\Delta$-regular graph $G$ with $\Delta\geq 1$
is always at most its edge domination number $\gamma_e(G)$,
which coincides with the domination number of its line graph.
We prove that
$\gamma(G)\leq \left(1+\frac{2(\Delta-1)}{\Delta 2^{\Delta}}\right)\gamma_e(G)$
for general $\Delta\geq 1$,
and
$\gamma(G)\leq \left(\frac{7}{6}-\frac{1}{204}\right)\gamma_e(G)$
for $\Delta=3$.
Furthermore, we verify our conjecture for cubic claw-free graphs.
\end{abstract}
{\small
\begin{tabular}{lp{13cm}}
{\bf Keywords:} & Domination; edge domination; minimum maximal matching\\
{\bf MSC 2010:} & 05C69, 05C70
\end{tabular}
}

\section{Introduction}

We consider finite, simple, and undirected graphs, and use standard terminology.
Let $G$ be a graph.
A set $D$ of vertices of $G$ is a {\it dominating set} in $G$
if every vertex in $V(G)\setminus D$ has a neighbor in $D$,
and the {\it domination number} $\gamma(G)$ of $G$
is the minimum cardinality of a dominating set in $G$.
For a set $M$ of edges of $G$,
let $V(M)$ denote the set of vertices of $G$
that are incident with an edge in $M$.
The set $M$ is a {\it matching} in $G$ if the edges in $M$
are pairwise disjoint, that is, $|V(M)|=2|M|$.
A matching $M$ in $G$ is {\it maximal} if it is maximal with respect to inclusion,
that is, the set $V(G)\setminus V(M)$ is independent.
Let the {\it edge domination number} $\gamma_e(G)$ of $G$
be the minimum size of a maximal matching in $G$.
A maximal matching in $G$ of size $\gamma_e(G)$ is a {\it minimum maximal matching}.

A natural connection between the domination number
and the edge domination number of a graph $G$
becomes apparent when considering the line graph $L(G)$ of $G$.
Since a maximal matching $M$ in $G$
is a maximal independent set in $L(G)$,
the edge domination number $\gamma_e(G)$ of $G$
equals the independent domination number $i(L(G))$ of $L(G)$.
Since $L(G)$ is always claw-free,
and since the independent domination number
equals the domination number in claw-free graphs \cite{alla},
$\gamma_e(G)$ actually
equals the domination number $\gamma(L(G))$ of $L(G)$.
While the domination number \cite{hahesl}
and the edge domination number \cite{yaga},
especially with respect to computational hardness
and algorithmic approximability \cite{caek,cafukopa,chch,golera,hoki,scvi},
have been studied extensively for a long time,
little seems to be known about their relation.
For regular graphs, we conjecture the following:

\begin{conjecture}\label{conjecture1}
If $G$ is a $\Delta$-regular graph with $\Delta\geq 1$,
then $\gamma(G)\leq\gamma_e(G)$.
\end{conjecture}
The conjecture is trivial for $\Delta\leq 2$,
and fails for non-regular graphs,
see Figure \ref{fig1}.
As pointed out by Felix Joos \cite{jo}, for $\Delta\geq 13$,
Conjecture \ref{conjecture1} follows by combining the known results
$\gamma(G)\leq \frac{(1+\ln(\Delta+1))n}{\Delta+1}$ (cf. \cite{alsp})
and
$\gamma_e(G)\geq \frac{\Delta n}{4\Delta-2}$ (cf.~(\ref{e1}) below),
that is, it is interesting for small values of $\Delta$ only.
Furthermore, he observed that the union of two triangles 
plus a perfect matching shows that Conjecture \ref{conjecture1} 
is tight for $\Delta=3$.
\begin{figure}[H]
\begin{center}
\begin{tikzpicture}

\node[fill, circle, inner sep=1.3pt] (v1) at (0,0) {};
\node[fill, circle, inner sep=1.3pt] (v2) at (1,0) {};
\node[fill, circle, inner sep=1.3pt] (v3) at (-0.5,0.5) {};

\node[fill, circle, inner sep=1.3pt] (v4) at (-.5,-0.5) {};
\node[fill, circle, inner sep=1.3pt] (v5) at (1.5,0.5) {};
\node[fill, circle, inner sep=1.3pt] (v6) at (1.5,-0.5) {};

\draw (v1) -- (v2);
\draw (v1) -- (v3);
\draw (v1) -- (v4);
\draw (v2) -- (v5);
\draw (v2) -- (v6);
\end{tikzpicture}
\end{center}
\caption{A non-regular graph $G$ with $\gamma(G)=2>1=\gamma_e(G)$.}\label{fig1}
\end{figure}
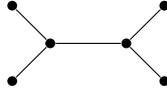
Our contributions
are three results related to Conjecture \ref{conjecture1}.
A simple probabilistic argument implies a weak version of Conjecture \ref{conjecture1},
which, for $\Delta\leq 12$,
is better than the above-mentioned consequence of \cite{alsp} and (\ref{e1}).
\begin{theorem}\label{theorem1}
If $G$ is a $\Delta$-regular graph with $\Delta\geq 1$,
then $\gamma(G)\leq \left(1+\frac{2(\Delta-1)}{\Delta 2^{\Delta}}\right)\gamma_e(G)$.
\end{theorem}
For cubic graphs, Theorem \ref{theorem1} implies $\gamma(G)\leq \frac{7}{6}\gamma_e(G)$,
which we improve with our next result.
Even though the improvement is rather small,
we believe that it is interesting
especially because of the approach used in its proof.

\begin{theorem}\label{theorem2}
If $G$ is a cubic graph,
then $\gamma(G)\leq \left(\frac{7}{6}-\frac{1}{204}\right)\gamma_e(G)$.
\end{theorem}
Finally, we show Conjecture \ref{conjecture1} for cubic claw-free graphs.

\begin{theorem}\label{theorem3}
If $G$ is a cubic claw-free graph,
then $\gamma(G)\leq \gamma_e(G)$.
\end{theorem}
All proofs are given in the following section.

\section{Proofs}
We begin with the simple probabilistic proof of Theorem \ref{theorem1},
which is also the basis for the proof of Theorem \ref{theorem2}.

\begin{proof}[Proof of Theorem \ref{theorem1}]
Let $M$ be a minimum maximal matching in $G$.
Since every vertex in $V(G)\setminus V(M)$ has $\Delta$ neighbors in $V(M)$,
and every vertex in $V(M)$ has at most $\Delta-1$ neighbors in $V(G)\setminus V(M)$,
we have
\begin{eqnarray}\label{e1}
\Delta(n-2\gamma_e(G))\leq 2(\Delta-1)\gamma_e(G),
\end{eqnarray}
where $n$ is the order of $G$.

Let the set $D$ arise by selecting,
for every edge in $M$,
one of the two incident vertices
independently at random with probability $1/2$.
Clearly, $|D|=\gamma_e(G)$.
If $u$ is a vertex in $V(G)\setminus V(M)$,
then $u$ has no neighbor in $D$
with probability at most $1/2^{\Delta}$.
Note that $u$ might be adjacent to both endpoints of some edge in $M$
in which case it always has a neighbor in $D$.
If $B$ is the set of vertices in $V(G)\setminus V(M)$
with no neighbor in $D$,
then linearity of expectation implies
$$\mathbb{E}[|B|]
=\sum\limits_{u\in V(G)\setminus V(M)}\mathbb{P}[u\in B]
\leq \frac{|V(G)\setminus V(M)|}{2^{\Delta}}
=\frac{n-2\gamma_e(G)}{2^{\Delta}}.$$
Since $D\cup B$ is a dominating set in $G$,
the first moment method implies
$$\gamma(G)\leq |D|+\mathbb{E}[|B|]
=\gamma_e(G)+\frac{n-2\gamma_e(G)}{2^{\Delta}}
\stackrel{(\ref{e1})}{\leq} \gamma_e(G)+\frac{2(\Delta-1)\gamma_e(G)}{\Delta 2^{\Delta}},$$
which completes the proof.
\end{proof}
The next proof arises by modifying the previous proof.

\begin{proof} [Proof of Theorem \ref{theorem2}]
Clearly, we may assume that $G$ is connected.
Let $M$ be a minimum maximal matching in $G$.
Let $R_0$ be the set of vertices from $V(G)\setminus V(M)$
that are adjacent to both endpoints of some edge in $M$,
and let $R$ be $(V(G)\setminus V(M))\setminus R_0$.
Also in this proof, we construct a random set $D$
containing exactly one vertex from every edge in $M$.
Note that every vertex from $R_0$ will always have a neighbor in $D$.
Again, let $B$ be the set of vertices in $R$ with no neighbor in $D$.
As before, we will use the estimate
$$\gamma(G)\leq \gamma_e(G)+\mathbb{E}[|B|]
=\gamma_e(G)+\sum_{u\in R}\mathbb{P}[u\in B].$$
Initially, we choose $D$ exactly as in the proof of Theorem \ref{theorem1},
which implies
$$\mathbb{E}[|B|]=\frac{|R|}{8}.$$
In order to obtain an improvement,
we iteratively modify the random choice of $D$
in such a way that $\mathbb{E}[|B|]$ becomes smaller.
We do this using two operations.
Each individual operation leads to some reduction of $\mathbb{E}[|B|]$,
and we ensure that all these reductions combine additively.
While the first operation leads to a reduction of $\mathbb{E}[|B|]$
regardless of additional structural properties of $G$,
our argument that the second operation leads to a reduction
is based on the assumption that the first operation
has been applied as often as possible.

The first operation is as follows.
\begin{itemize}
\item If there are two edges $uv$ and $u'v'$ in $M$ such that
the set $X$ of vertices $x$ in $R$ with
$$N_G(x)\cap \{ u,v,u',v'\}\in \big\{ \{ u,u'\},\{v,v'\}\big\}$$
is larger than
the set $Y$ of vertices $y$ in $R$ with
$$N_G(y)\cap \{ u,v,u',v'\}\in \big\{ \{ u,v'\}, \{v,u'\}\big\},$$
see Figure \ref{figcp},
then we {\it couple} the random choices for the pair $\{ uv, u'v'\}$
in such a way that $D$ contains $\{ u,v'\}$ with probability $1/2$
and $\{ u',v\}$ with probability $1/2$.
\end{itemize}
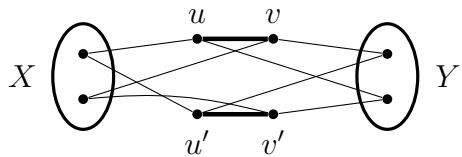
\begin{figure}[H]
\begin{center}
\begin{tikzpicture}

\node[fill, circle, inner sep=1.3pt, label=above:{$u$}] (u1) at (0,3) {};
\node[fill, circle, inner sep=1.3pt, label=above:{$v$}] (u2) at (1,3) {};
\node[fill, circle, inner sep=1.3pt, label=below:{$u'$}] (u3) at (0,2) {};
\node[fill, circle, inner sep=1.3pt, label=below:{$v'$}] (u4) at (1,2) {};

\draw[line width=0.6mm] (u1) -- (u2);
\draw[line width=0.6mm] (u4) -- (u3);

\node (h1) at (-2.3,2.5) {$X$};
\node (h2) at (3.3,2.5) {$Y$};

\node[fill, circle, inner sep=1.3pt] (z1) at (2.5,2.2) {};
\node[fill, circle, inner sep=1.3pt] (z2) at (2.5,2.8) {};
\node[fill, circle, inner sep=1.3pt] (z3) at (-1.5,2.8) {};
\node[fill, circle, inner sep=1.3pt] (z4) at (-1.5,2.2) {};

\draw (z1) -- (u1);
\draw (z1) -- (u4);
\draw (z2) -- (u2);
\draw (z2) -- (u3);

\draw (z3) -- (u1);
\draw (z3) -- (u3);
\draw (z4) to (u2);
\draw (z4) to[ bend left=10] (u4);

\draw[very thick]
        (-1.5,2.5) ellipse (0.4 and 0.7)
        (2.5,2.5) ellipse (0.4 and 0.7);

\end{tikzpicture}
\end{center}
\caption{The edges $uv$, $u'v'$ and the sets $X$ and $Y$.}
\label{figcp}
\end{figure}
The choice for the {\it coupled pair} $\{ uv,u'v'\}$
will remain independent of all other random choices
involved in the construction of $D$.
Furthermore, the two edges in a coupled pair
will not be involved in any other operation
modifying the choice of $D$.

Let $\pi$ be a coupled pair $\{ uv,u'v'\}$.
By construction, we obtain $\mathbb{P}[x\in B]=0$ for every vertex $x$ in $X$.
Now, consider a vertex $y$ in $Y$.
The two neighbors of $y$ in the two coupled edges
are either both in $D$ or both outside of $D$,
each with probability exactly $1/2$.
We will ensure that the third neighbor of $y$,
which is necessarily in a third edge from $M$,
will belong to $D$ still with probability exactly $1/2$.
By the independence mentioned above, we have $\mathbb{P}[y\in B]=1/4$.
Recall that, for the choice of $D$ as in the proof of Theorem \ref{theorem1},
each vertex from $X\cup Y$ belongs to $B$ with probability exactly $1/8$.
Hence, by coupling the pair $\pi$, the expected cardinality $\mathbb{E}[|B|]$ of $B$ is reduced by $(|X|-|Y|)/8$,
which is at least $1/8$.

The second operation is as follows.
\begin{itemize}
\item We select a suitable vertex $z$ from $R$
such that it has no neighbor in any of the coupled edges.
If the edges $u_1v_1$, $u_2v_2$, and $u_3v_3$ from $M$ are such that $u_1$, $u_2$, and $u_3$
are the three neighbors of $z$,
then we {\it derandomize} the selection for these three edges,
and $D$ will always contain $u_1$, $u_2$, and $u_3$.
We call $\{ u_1v_1,u_2v_2,u_3v_3\}$ a {\it derandomized triple with center $z$}.
\end{itemize}
We will first couple a maximal number of pairs,
and then derandomize triples one after the other
as long as possible.

Let $\tau=\{ u_1v_1,u_2v_2,u_3v_3\}$
be the next triple to be derandomized at some point.
Let $S(\tau)$ be the set of all vertices
that are incident with an edge $e$ from $M\setminus \tau$
such that some vertex in $R$ has a neighbor in $V(\tau)$
as well as in $e$, see Figure \ref{figstau}.

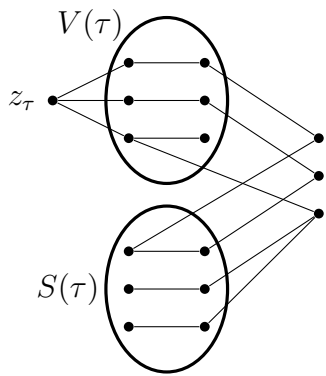
\begin{figure}[H]
\begin{center}
\begin{tikzpicture}

\node[fill, circle, inner sep=1.3pt, label=left:{$z_\tau$}] (z) at (-1,2.5) {};

\node[fill, circle, inner sep=1.3pt] (u1) at (0,2) {};
\node[fill, circle, inner sep=1.3pt] (u2) at (1,2) {};
\node[fill, circle, inner sep=1.3pt] (u3) at (0,2.5) {};
\node[fill, circle, inner sep=1.3pt] (u4) at (1,2.5) {};
\node[fill, circle, inner sep=1.3pt] (u5) at (0,3) {};
\node[fill, circle, inner sep=1.3pt] (u6) at (1,3) {};

\draw (u1) -- (u2);
\draw (u4) -- (u3);
\draw (u5) -- (u6);
\draw (z) -- (u1);
\draw (z) -- (u3);
\draw (z) -- (u5);

\node[fill, circle, inner sep=1.3pt] (w1) at (0,0.5) {};
\node[fill, circle, inner sep=1.3pt] (w2) at (1,0.5) {};
\node[fill, circle, inner sep=1.3pt] (w3) at (0,0) {};
\node[fill, circle, inner sep=1.3pt] (w4) at (1,0) {};
\node[fill, circle, inner sep=1.3pt] (w5) at (0,-.5) {};
\node[fill, circle, inner sep=1.3pt] (w6) at (1,-.5) {};

\draw (w1) -- (w2);
\draw (w4) -- (w3);
\draw (w5) -- (w6);

\node (h1) at (-0.8,0) {$S(\tau)$};
\node (h2) at (-0.5,3.5) {$V(\tau)$};

\node[fill, circle, inner sep=1.3pt] (z1) at (2.5,2) {};
\node[fill, circle, inner sep=1.3pt] (z2) at (2.5,1.5) {};
\node[fill, circle, inner sep=1.3pt] (z3) at (2.5,1) {};

\draw (z1) -- (u6);
\draw (z1) -- (w1);
\draw (z2) -- (u4);
\draw (z2) -- (w2);
\draw (z3) -- (u1);
\draw (z3) -- (w4);
\draw (z3) -- (w6);

\draw[very thick]
        (0.5,0) ellipse (0.8 and 1.1)
        (.5,2.5) ellipse (0.8 and 1.1);

\end{tikzpicture}
\end{center}
\caption{The set $S(\tau)$.}
\label{figstau}
\end{figure}
During all changes of the initial random choice of $D$ performed so far,
we ensure that the following property holds
just before we derandomize the triple $\tau$:
\begin{eqnarray}\label{eproptau}
\begin{minipage}{0.8\textwidth}
\it For every vertex $u$ in $R$ that has a neighbor in $V(\tau)\cup S(\tau)$,
the three neighbors of $u$ in $V(M)$
belong to $D$ independently with probability $1/2$.
\end{minipage}
\end{eqnarray}
All coupled pairs and derandomized triples will be disjoint.

For every edge in $M$ that does not belong to any coupled pair or derandomized triple,
we select the endpoint that is added to $D$ exactly as in the proof of Theorem \ref{theorem1},
that is, with probability $1/2$ independently of all other random choices
involved in the construction of $D$.

\bigskip

\noindent We fix a maximal collection ${\cal P}$
of pairwise disjoint coupled pairs $\pi_1,\ldots,\pi_p$.

Let $S_{\rm paired}$ be the set of the $4p$ vertices from $V(M)$
that are incident with some of the $2p$ paired edges.
Let $R_1$ be the set of vertices in $R$ with exactly one neighbor in $S_{\rm paired}$,
and let $R_2$ be the set of vertices in $R$ with at least two neighbors in $S_{\rm paired}$.
Note that the sets $R_0$, $R_1$, and $R_2$ are disjoint by definition.
Let $S_{\rm paired}'$ be the set of vertices from $V(M)\setminus S_{\rm paired}$
that are incident with an edge in $M$ that contains a neighbor of some vertex in $R_0\cup R_2$.
Let $R_3$ be the set of vertices from $R\setminus (R_0\cup R_1\cup R_2)$
that have a neighbor in $S_{\rm paired}'$.
All sets are illustrated in Figure \ref{fig3}.
Let $$R^{(1)}=R\setminus (R_0\cup R_1\cup R_2\cup R_3),$$
$r=|R|$, $r^{(1)}=|R^{(1)}|$, and $r_i=|R_i|$ for $i\in \{ 0,1,2,3\}$.
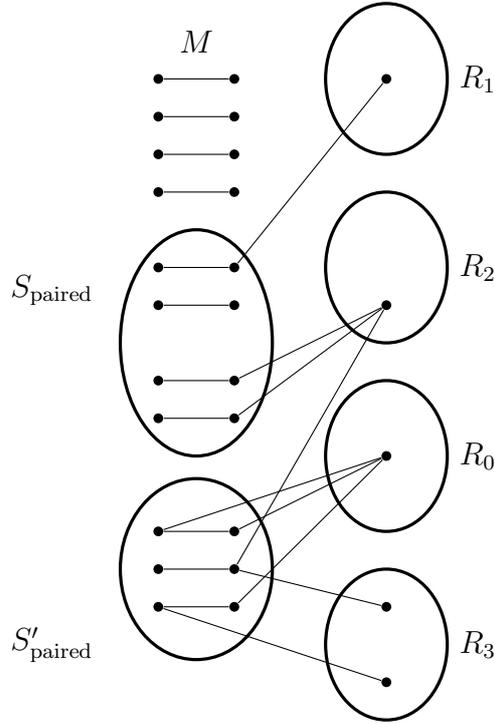
\begin{figure}[H]
\begin{center}
\begin{tikzpicture}

\node[fill, circle, inner sep=1.3pt] (v1) at (0,0) {};
\node[fill, circle, inner sep=1.3pt] (v2) at (1,0) {};
\node[fill, circle, inner sep=1.3pt] (v3) at (0,0.5) {};
\node[fill, circle, inner sep=1.3pt] (v4) at (1,.5) {};
\node[fill, circle, inner sep=1.3pt] (v5) at (0,-1) {};
\node[fill, circle, inner sep=1.3pt] (v6) at (1,-1) {};
\node[fill, circle, inner sep=1.3pt] (v7) at (0,-1.5) {};
\node[fill, circle, inner sep=1.3pt] (v8) at (1,-1.5) {};
\draw (v1) -- (v2);
\draw (v4) -- (v3);
\draw (v5) -- (v6);
\draw (v7) -- (v8);

\node[fill, circle, inner sep=1.3pt] (u1) at (0,2) {};
\node[fill, circle, inner sep=1.3pt] (u2) at (1,2) {};
\node[fill, circle, inner sep=1.3pt] (u3) at (0,2.5) {};
\node[fill, circle, inner sep=1.3pt] (u4) at (1,2.5) {};
\node[fill, circle, inner sep=1.3pt] (u5) at (0,3) {};
\node[fill, circle, inner sep=1.3pt] (u6) at (1,3) {};
\node[fill, circle, inner sep=1.3pt] (u7) at (0,1.5) {};
\node[fill, circle, inner sep=1.3pt] (u8) at (1,1.5) {};

\draw (u1) -- (u2);
\draw (u4) -- (u3);
\draw (u5) -- (u6);
\draw (u7) -- (u8);

\node[fill, circle, inner sep=1.3pt] (w1) at (0,-3.5) {};
\node[fill, circle, inner sep=1.3pt] (w2) at (1,-3.5) {};
\node[fill, circle, inner sep=1.3pt] (w3) at (0,-3) {};
\node[fill, circle, inner sep=1.3pt] (w4) at (1,-3) {};
\node[fill, circle, inner sep=1.3pt] (w5) at (0,-4) {};
\node[fill, circle, inner sep=1.3pt] (w6) at (1,-4) {};
\draw (w1) -- (w2);
\draw (w4) -- (w3);
\draw (w5) -- (w6);

\draw[very thick]
        (0.5,-0.5) ellipse (1 and 1.5)
        (0.5,-3.5) ellipse (1 and 1.2)
        (3,3) ellipse (.8 and 1)
        (3,-2) ellipse (.8 and 1)
        (3,0.5) ellipse (.8 and 1)
        (3,-4.5) ellipse (.8 and 1);
\node (h1) at (-1.4,0.2) {$S_{\rm paired}$};
\node (h2) at (-1.4,-4.5) {$S_{\rm paired}'$};
\node (h3) at (0.5,3.5) {$M$};
\node (h4) at (4.2,0.5) {$R_2$};
\node (h5) at (4.2,3) {$R_1$};
\node (h6) at (4.2,-2) {$R_0$};
\node (h7) at (4.2,-4.5) {$R_3$};

\node[fill, circle, inner sep=1.3pt] (z2) at (3,3) {};
\node[fill, circle, inner sep=1.3pt] (z4) at (3,0) {};
\node[fill, circle, inner sep=1.3pt] (z1) at (3,-2) {};
\node[fill, circle, inner sep=1.3pt] (z5) at (3,-5) {};
\node[fill, circle, inner sep=1.3pt] (z6) at (3,-4) {};
\draw (z1) -- (w3);
\draw (z1) -- (w6);
\draw (z1) -- (w4);
\draw (z2) -- (v4);
\draw (z4) -- (w2);
\draw (z4) -- (v6);
\draw (z4) -- (v8);
\draw (z5) -- (w5);
\draw (z6) -- (w2);

\end{tikzpicture}
\end{center}
\caption{The sets $S_{\rm paired}$, $S_{\rm paired}'$, $R_0$, $R_1$, $R_2$, and $R_3$}\label{fig3}
\end{figure}

Since $G$ has at most $8p$ edges leaving $S_{\rm paired}$,
we have $2r_2+r_1\leq 8p$,
which implies $r_1+7r_2\leq 28p$.
By definition, we obtain $|S_{\rm paired}'|\leq 4r_0+2r_2$.
Considering the number of edges leaving $S_{\rm paired}'$,
we obtain $r_3\leq 3|S_{\rm paired}'|\leq 12r_0+6r_2$.
Therefore,
\begin{eqnarray}
r^{(1)} & = & r-r_0-r_1-r_2-r_3\nonumber\\
& \geq & r-13r_0-r_1-7r_2\nonumber\\
& \geq & r-13r_0-28p.\label{er1}
\end{eqnarray}
Note that, only coupling the pairs $\pi_1,\ldots,\pi_p$
and not derandomizing any triple, we have
\begin{eqnarray}\label{e2}
\mathbb{E}[|B|]\leq \frac{|R|}{8}-\frac{p}{8}
=\frac{1}{8}\left(n-2\gamma_e(G)-r_0\right)-\frac{p}{8}.
\end{eqnarray}
If $r_0+p$ is large enough,
then this already yields the desired improvement.
Since we cannot guarantee this,
we now form derandomized triples one by one
with centers from $R^{(1)}$.
For every selected triple to be derandomized,
we remove suitable vertices from $R^{(1)}$
in order to ensure (\ref{eproptau}).
Suppose that we have already formed $t-1$ such derandomized triples
with centers $z_1,\ldots,z_{t-1}$,
then the center $z_t$ for the triple $\tau_t$
will be selected from $R^{(t)}$,
where $t$ is initially $1$,
and $R^{(t+1)}$ is obtained from $R^{(t)}$
by removing every vertex from $R^{(t)}$
that has a neighbor in $V(\tau_t)\cup S(\tau_t)$.
This ensures that
all coupled pairs and derandomized triples are disjoint
as well as (\ref{eproptau}).

Now, we analyze the reduction of $\mathbb{E}[|B|]$,
or rather the reduction of the upper bound on $\mathbb{E}[|B|]$ given in (\ref{e2}),
incurred by some derandomized triple $\tau_t$ with center $z_t$.
Let $e_1$, $e_2$, and $e_3$ in $M$ be such that $e_i=u_iv_i$ for $i\in [3]$
and $z_t$ is adjacent to $u_1$, $u_2$, and $u_3$,
that is, $\tau_t=\{ u_1v_1,u_2v_2,u_3v_3\}$.

We consider two cases.

\bigskip

\noindent {\bf Case 1} {\it Some vertex $z$ in $R$ distinct from $z_t$
has three neighbors in $V(\tau_t)$.}

\bigskip

\noindent First, suppose that $z$ is adjacent to $u_1$ and $u_2$.
In this case,
the pair $e_1$ and $e_2$ could be coupled and added to ${\cal P}$,
contradicting the choice of ${\cal P}$.
Next, suppose that $z$ is adjacent to $v_1$, $v_2$, and $v_3$.
Since the pair $e_1$ and $e_2$ cannot be coupled and added to ${\cal P}$,
there are two vertices $z'$ and $z''$ in $R$ such that
$z'$ is adjacent to $u_1$ and $v_2$,
and
$z''$ is adjacent to $u_2$ and $v_1$.
Since the pair $e_2$ and $e_3$ cannot be coupled and added to ${\cal P}$,
the vertex $z'$ is adjacent to $u_3$,
which implies the contradiction
that the pair $e_1$ and $e_3$ could be coupled and added to ${\cal P}$.

Hence, by symmetry,
we may assume that $z$ is adjacent to $u_1$, $v_2$, and $v_3$.
Since the pair $e_2$ and $e_3$ cannot be coupled and added to ${\cal P}$,
there are two vertices $z'$ and $z''$ in $R$ such that
$z'$ is adjacent to $u_2$ and $v_3$,
and
$z''$ is adjacent to $u_3$ and $v_2$.
If $z''$ is adjacent to $v_1$,
then, considering the pair $e_1$ and $e_2$,
it follows that $z'$ must be adjacent to $v_1$.
In this case, the connected graph $G$ has order $10$,
and $\{ u_1,u_2,u_3\}$ is a dominating set,
which implies the statement.
Hence, we may assume that $z''$ is not adjacent to $v_1$.
A symmetric argument implies that $z'$ is not adjacent to $v_1$.
See Figure \ref{figcase1} for an illustration.
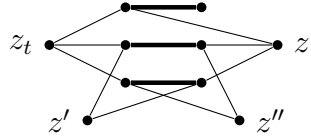
\begin{figure}[H]
\begin{center}
\begin{tikzpicture}

\node[fill, circle, inner sep=1.3pt, label=left:{$z_t$}] (z) at (-1,2.5) {};

\node[fill, circle, inner sep=1.3pt] (u1) at (0,3) {};
\node[fill, circle, inner sep=1.3pt] (u2) at (1,3) {};
\node[fill, circle, inner sep=1.3pt] (u3) at (0,2.5) {};
\node[fill, circle, inner sep=1.3pt] (u4) at (1,2.5) {};
\node[fill, circle, inner sep=1.3pt] (u5) at (0,2) {};
\node[fill, circle, inner sep=1.3pt] (u6) at (1,2) {};

\draw[line width=0.6mm] (u1) -- (u2);
\draw[line width=0.6mm] (u4) -- (u3);
\draw[line width=0.6mm] (u5) -- (u6);
\draw (z) -- (u1);
\draw (z) -- (u3);
\draw (z) -- (u5);

\node[fill, circle, inner sep=1.3pt, label=right:{$z$}] (z1) at (2,2.5) {};
\node[fill, circle, inner sep=1.3pt, label=left:{$z'$}] (z2) at (-.5,1.5) {};
\node[fill, circle, inner sep=1.3pt, label=right:{$z''$}] (z3) at (1.5,1.5) {};

\draw (z1) -- (u1);
\draw (z1) -- (u4);
\draw (z1) -- (u6);
\draw (z2) -- (u3);
\draw (z2) -- (u6);
\draw (z3) -- (u5);
\draw (z3) -- (u4);

%

\end{tikzpicture}
\end{center}
\caption{The edges in $\tau$ and the vertices $z_t$, $z$, $z'$, and $z''$.}
\label{figcase1}
\end{figure}
Our derandomized choice of adding always $u_1$, $u_2$, and $u_3$ to $D$ yields
$$\mathbb{P}[z_t \in B]
=\mathbb{P}[z\in B]
=\mathbb{P}[z'\in B]
=\mathbb{P}[z''\in B]=0.$$
Furthermore, property (\ref{eproptau}) implies
$\mathbb{P}[w\in B]=1/4$ for every neighbor $w$ of $v_1$ in $R$.
Since $v_1$ has at most two such neighbors,
derandomizing the triple $\tau_t$
additionally reduces the upper bound on $\mathbb{E}[|B|]$ given in (\ref{e2})
by at least $\frac{4}{8}-\frac{2}{8}=\frac{1}{4}$.
Since $z'$ and $z''$ both have at most one neighbor not in $V(\tau_t)$,
and at most two neighbors of $v_1$ in $R$
both have at most two neighbors not in $V(\tau_t)$,
we obtain $|S(\tau_t)|\leq 12$, and
\begin{eqnarray*}
|R^{(t+1)}|
& = |R^{(t)}|
& -\,\, \big|\big\{v\in R^{(t)}:v \mbox{ has a neighbor in }V(\tau_t)\cup S(\tau_t)\big\}\big|\\
& = |R^{(t)}|
& -\,\, \big|\big\{v\in R^{(t)}:v \mbox{ has a neighbor in }V(\tau_t)\big\}\big|\\
&&
-\,\, \big|\big\{v\in R^{(t)}:v \mbox{ has a neighbor in }S(\tau_t)
\mbox{ but no neighbor in }V(\tau_t)\big\}\big|\\
&\geq |R^{(t)}|&-\,\, 6-3\cdot 6\\
&=|R^{(t)}|&-\,\, 24.
\end{eqnarray*}

\bigskip

\noindent {\bf Case 2} {\it $z_t$ is the only vertex in $R$
that has three neighbors in $V(\tau_t)$.}

\bigskip

\noindent Since the pair $e_1$ and $e_2$ cannot be coupled and added to ${\cal P}$,
we may assume, by symmetry,
that there is a vertex $z$ in $R$ that is adjacent to $u_1$ and $v_2$.
Since the pair $e_2$ and $e_3$ cannot be coupled and added to ${\cal P}$,
we may assume
that there is a vertex $z'$ in $R$ such that
either
$z'$ is adjacent to $u_3$ and $v_2$
or
$z'$ is adjacent to $u_2$ and $v_3$.

If $z'$ is adjacent to $u_3$ and $v_2$,
then the assumption of Case 2 implies the contradiction
that the pair $e_1$ and $e_3$ can be coupled and added to ${\cal P}$.
Hence, we may assume that $z'$ is adjacent to $u_2$ and $v_3$.
Since the pair $e_1$ and $e_3$ cannot be coupled and added to ${\cal P}$,
there is a vertex $z''$ in $R$ adjacent to $u_3$ and $v_1$.
See Figure \ref{figcase2} for an illustration.
\begin{figure}[H]
\begin{center}
\begin{tikzpicture}

\node[fill, circle, inner sep=1.3pt, label=left:{$z_\tau$}] (z) at (-1,2.5) {};

\node[fill, circle, inner sep=1.3pt] (u1) at (0,3) {};
\node[fill, circle, inner sep=1.3pt] (u2) at (1,3) {};
\node[fill, circle, inner sep=1.3pt] (u3) at (0,2.5) {};
\node[fill, circle, inner sep=1.3pt] (u4) at (1,2.5) {};
\node[fill, circle, inner sep=1.3pt] (u5) at (0,2) {};
\node[fill, circle, inner sep=1.3pt] (u6) at (1,2) {};

\draw[line width=0.6mm] (u1) -- (u2);
\draw[line width=0.6mm] (u4) -- (u3);
\draw[line width=0.6mm] (u5) -- (u6);
\draw (z) -- (u1);
\draw (z) -- (u3);
\draw (z) -- (u5);

\node[fill, circle, inner sep=1.3pt, label=right:{$z$}] (z1) at (1.5,3.5) {};
\node[fill, circle, inner sep=1.3pt, label=left:{$z'$}] (z2) at (-.5,1.5) {};
\node[fill, circle, inner sep=1.3pt, label=right:{$z''$}] (z3) at (1.5,1.5) {};

\draw (z1) -- (u1);
\draw (z1) -- (u4);
\draw (z2) -- (u3);
\draw (z2) -- (u6);
\draw (z3) -- (u5);
\draw (z3) -- (u2);

%

\end{tikzpicture}
\end{center}
\caption{The edges in $\tau$ and the vertices $z_t$, $z$, $z'$, and $z''$.}
\label{figcase2}
\end{figure}
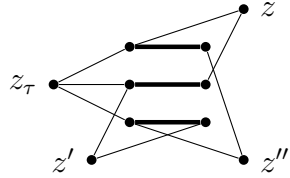
The choice of ${\cal P}$ implies that no vertex from $R^{(t)}$ distinct from $z_t$, $z$, $z'$, and $z''$
has two neighbors in $V(\tau_t)$.
Arguing as above, we obtain that
derandomizing the triple $\tau_t$
additionally reduces the upper bound on $\mathbb{E}[|B|]$ given in (\ref{e2})
by at least $\frac{4}{8}-\frac{3}{8}=\frac{1}{8}$.
Similarly as in Case 1, it follows that $|S(\tau_t)|\leq 18$, and that
$$|R^{(t+1)}|=|R^{(t)}\setminus\{v\in R:~v \mbox{ has a neighbor in }V(\tau_t)\cup S(\tau_t)\}| \geq |R^{(t)}| -7-3\cdot 9= |R^{(t)}|-34.$$

\bigskip

\noindent Since we derandomize as many triples as possible,
it follows that the number $t$ of derandomized triples satisfies
$$t\geq \frac{r^{(1)}}{34}\stackrel{(\ref{er1})}{\geq} \frac{r-13r_0-28p}{34},$$
and that
the joint reduction of the upper bound on $\mathbb{E}[|B|]$ given in (\ref{e2})
is at least
$$\frac{t}{8}\geq \frac{r-13r_0-28p}{272}.$$
Altogether, coupling all $p$ pairs in ${\cal P}$, and derandomizing the $t$ triples,
we obtain
\begin{eqnarray*}
\mathbb{E}[|B|]
&\leq & \frac{1}{8}\left(n-2\gamma_e(G)-r_0\right)-\frac{p}{8}-\frac{t}{8}\\
&\leq & \frac{1}{8}\left(n-2\gamma_e(G)-r_0-p\right)-\frac{r-13r_0-28p}{272}\\
& = & \frac{1}{8}\left(n-2\gamma_e(G)-r_0-p\right)-\frac{n-2\gamma_e(G)-r_0-13r_0-28p}{272}\\
& = &
\frac{33}{272}(n-2\gamma_e(G))
-\frac{5}{68}r_0
-\frac{3}{136}p\\
& \leq &
\frac{33}{272}(n-2\gamma_e(G))\\
& \stackrel{(\ref{e1})}{\leq} &
\frac{11}{68}\gamma_e(G).
\end{eqnarray*}
Therefore,
$$\gamma(G)\leq \gamma_e(G)+\mathbb{E}[|B|]
\leq \frac{79}{68}\gamma_e(G)
=\left(\frac{7}{6}-\frac{1}{204}\right)\gamma_e(G),$$
which completes the proof.
\end{proof}
We proceed to the final proof.

\begin{proof} [Proof of Theorem \ref{theorem3}]
Let $M$ be a minimum maximal matching in $G$.
Let the set $D$ of $|M|$ vertices intersecting each edge in $M$
be chosen such that the set
$B=\{ u\in V(G)\setminus V(M):|N_G(u)\cap D|=0\}$
is smallest possible.
For a contradiction, we may suppose that $B$ is non-empty.
Let
$C=\{ u\in V(G)\setminus V(M):|N_G(u)\cap D|=1\}.$
Let $b$ be a vertex in $B$.
Let $u_{-1}v_{-1}$, $u_0v_0$, and $u_1v_1$ in $M$ be such that $N_G(b)=\{ v_{-1},v_0,v_1\}$.
Since $D$ intersects each edge in $M$,
we have $u_{-1},u_0,u_1\in D$.
Since $G$ is claw-free, we may assume, by symmetry,
that $v_0$ and $v_1$ are adjacent,
which implies that $v_{-1}$ is not adjacent to $v_0$ or $v_1$.
Let $x$ be the neighbor of $v_{-1}$ distinct from $u_{-1}$ and $b$.
Since $G$ is claw-free, the vertex $x$ is adjacent to $u_{-1}$.
If $x=u_0$, then $u_0$ has no neighbor in $C$,
and exchanging $u_0$ and $v_0$ within $D$ reduces $|B|$,
which is a contradiction.
Hence, by symmetry between $u_0$ and $u_1$, the vertex $x$ is distinct from $u_0$ and $u_1$.
Since exchanging $u_1$ and $v_1$ within $D$ does not reduce $|B|$,
the vertex $u_1$ has a neighbor $c_1$ in $C$,
which is necessarily distinct from $x$.

Now, let
$\sigma:v_1,u_1,c_1,v_2,u_2,c_2,\ldots,v_k,u_k,c_k$
be a maximal sequence of distinct vertices from
$V(G)\setminus \{ u_{-1},u_0,v_{-1},v_0,b,x\}$
such that
$u_iv_i\in M$, $u_i\in D$, $c_i\in C$,
$u_i$ is adjacent to $c_i$ for every $i\in [k]$,
and $v_{i+1}$ is adjacent to $u_i$ for every $i\in [k-1]$.
Let $X=\{ u_{-1},u_0,v_{-1},v_0,b,x\}\cup \{ v_1,u_1,c_1,v_2,u_2,c_2,\ldots,v_k,u_k,c_k\}$,
and see Figure \ref{figseq} for an illustration.

\begin{figure}[H]
\begin{center}
\begin{tikzpicture}[scale=0.7]

\node[fill, circle, inner sep=1.3pt, label=left:{$x$}] (x) at (0,1) {};
\node[fill, circle, inner sep=1.3pt, label=below left:{$u_{-1}$}] (u-1) at (1,0) {};
\node[fill, circle, inner sep=1.3pt, label=above left:{$v_{-1}$}] (v-1) at (1,2) {};
\node[fill, circle, inner sep=1.3pt, label=above:{$b$}] (b) at (2,3) {};
\node[fill, circle, inner sep=1.3pt, label=below:{$u_0$}] (u0) at (2,0) {};
\node[fill, circle, inner sep=1.3pt] (v0) at (2,2) {};
\node[fill, circle, inner sep=1.3pt, label=below:{$u_1$}] (u1) at (3,0) {};
\node[fill, circle, inner sep=1.3pt, label=above right:{$v_1$}] (v1) at (3,2) {};
\node[fill, circle, inner sep=1.3pt, label=above left:{$u_2$}] (u2) at (5,2) {};
\node[fill, circle, inner sep=1.3pt, label=below right:{$v_2$}] (v2) at (5,0) {};
\node[fill, circle, inner sep=1.3pt, label=below left:{$u_3$}] (u3) at (7,0) {};
\node[fill, circle, inner sep=1.3pt, label=above right:{$v_3$}] (v3) at (7,2) {};
\node[fill, circle, inner sep=1.3pt, label=above left:{$u_4$}] (u4) at (9,2) {};
\node[fill, circle, inner sep=1.3pt, label=below right:{$v_4$}] (v4) at (9,0) {};

\node[fill, circle, inner sep=1.3pt, label=below:{$c_1$}] (c1) at (4,-1) {};
\node[fill, circle, inner sep=1.3pt, label=above:{$c_2$}] (c2) at (6,3) {};
\node[fill, circle, inner sep=1.3pt, label=below:{$c_3$}] (c3) at (8,-1) {};
\node[fill, circle, inner sep=1.3pt, label=above:{$c_4$}] (c4) at (10,3) {};

\draw (x) -- (u-1);
\draw (x) -- (v-1);
\draw[line width=0.6mm] (u-1) -- (v-1);
\draw[line width=0.6mm] (u0) -- (v0);
\draw (v0) -- (v1);
\draw[line width=0.6mm] (u1) -- (v1);
\draw (b) -- (v-1);
\draw (b) -- (v0);
\draw (b) -- (v1);

\draw (u1) -- (c1);
\draw (u1) -- (v2);
\draw (c1) -- (v2);
\draw[line width=0.6mm]  (v2) -- (u2);

\draw (u2) -- (c2);
\draw (u2) -- (v3);
\draw (c2) -- (v3);
\draw[line width=0.6mm] (v3) -- (u3);

\draw (u3) -- (c3);
\draw (u3) -- (v4);
\draw (c3) -- (v4);
\draw[line width=0.6mm] (v4) -- (u4);

\draw (u4) -- (c4);
\end{tikzpicture}
\end{center}
\caption{A subgraph of $G$ with vertex set $X$, where $k=4$.}
\label{figseq}
\end{figure}
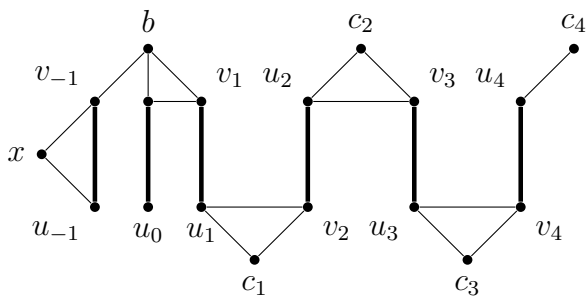
Let $v_{k+1}$ be the neighbor of $u_k$ distinct from $v_k$ and $c_k$.
Since $G$ is claw-free, the vertex $v_{k+1}$ is adjacent to $c_k$.
Since $V(G)\setminus V(M)$ is independent,
we have $u_{k+1}v_{k+1}\in M$ for some vertex $u_{k+1}$.
Since $c_k\in C$ and $u_k\in D$,
we obtain $v_{k+1}\not\in D$ and $u_{k+1}\in D$,
which implies that
the vertex $v_{k+1}$ does not belong to $X$.

If $u_{k+1}$ belongs to $X$, then $u_{k+1}=x$,
and replacing $D$ with
$$D'=(D\setminus \{ u_1,u_2,\ldots,u_{k+1}\})\cup \{ v_1,v_2,\ldots,v_{k+1}\}$$
reduces $|B|$, which is a contradiction.
Hence, the vertex $u_{k+1}$ does not belong to $X$.
If $u_{k+1}$ has a neighbor $c_{k+1}$ in $C$,
then, by the structural conditions,
the vertex $c_{k+1}$ does not belong to $X$,
and the sequence $\sigma$ can be extended by appending $v_{k+1},u_{k+1},c_{k+1}$,
contradicting its choice.
Hence, the vertex $u_{k+1}$ has no neighbor in $C$,
and replacing $D$ with the set $D'$ as above
again reduces $|B|$.
This final contradiction completes the proof.
\end{proof}
\noindent {\bf Acknowledgement} We thank Felix Joos 
for pointing out that Conjecture \ref{conjecture1}
holds for large values of $\Delta$.

\end{document}